\documentclass{article}
\usepackage[utf8]{inputenc}

\usepackage{graphicx,graphics}

\usepackage{rotating}
\usepackage[twoside]{geometry}
\usepackage{epsfig}
\usepackage{indentfirst}
\usepackage[usenames,dvipsnames,svgnames,table]{xcolor}
\usepackage{cmap}

\usepackage{graphicx,graphics}

\usepackage{hyperref}
\hypersetup{
pdftitle={Semigroups of operators},    
pdfauthor={Moacir Aloisio},    
pdfnewwindow=true,      
colorlinks=true,      
linkcolor=black,         
citecolor=blue}

\usepackage{amssymb,amsfonts,amstext,amsthm}
\usepackage{mathrsfs}
\usepackage[intlimits]{amsmath}
\newtheorem{definition}{\bf Definition}[section]
\newtheorem{proposition}{\bf Proposition}[section]

\newtheorem{lemma}{\bf Lemma}[section]
\newtheorem{theorem}{\bf Theorem}[section]

\newtheorem{remark}{Remark}[section]

\newcounter{A}
\newenvironment{A}{\refstepcounter{A}\equation}{\tag{A}\endequation}
\newcounter{B}
\newenvironment{B}{\refstepcounter{B}\equation}{\tag{B}\endequation}

\title{A note on spectrum and quantum dynamics}
\author{Moacir Aloisio}
\date{May 2019}

\begin{document}

\maketitle

\begin{abstract} We show, in the same vein of Simon's Wonderland Theorem, that, typically in Baire's sense, the rates with whom the solutions of the Schr\"odinger equation escape, in time average, from every finite-dimensional subspace, depend on subsequences of time going to infinite.
\end{abstract} 

\


\

\noindent{\bf  AMS classification codes}: 46L60 (primary), 81Q10 (secondary).

\renewcommand{\thetable}{\Alph{table}}

%%%%%%%%%%%%%%%%%%%%%%%%%%%%%%%%%%%%%%%%%%%%%%%%%%%%%%%%%%%%%%%%%%%%%%%%%%%--Introduction--%%%%%%%%%%%%%%%%%%%%%%%%%%%%%%%%%%%%%%%%%%%%%%%%%%%%%%%%%%%%%%%%%%%%%%%%%%%%%%%%%%

\section{Introduction and results} 

\noindent There is a vast literature concerning the large time asymptotic behaviour of wave packet solutions of the Schr\"odinger equation
\begin{equation}\label{SE}
\begin{cases} \partial_t \xi = -iT\xi, ~t \in {\mathbb{R}}, \\ \xi(0) = \xi, ~~~\xi \in {\cal{H}},  \end{cases}    
\end{equation}
where $T$ is a self-adjoint operator in a separable complex Hilbert space $\cal{H}$. Namely, the relations between the quantum dynamics of solutions of (\ref{SE}) and the spectral properties of $T$ are a classical subject of quantum mechanics. In this context, we refer to \cite{Aloisio,Barbaroux1,Barbaroux2,Carvalho1,Carvalho2,DamanikDDP,Germinet,Guarneri1,Guarneri2,Last,StolzIntr}, among others. 

%%%%%%%%%%%%%%%%%%%%%%%%%%%%%%%%%%%%%%%%%%%%%%%%%%%%%%%%%%%%%%%%%%%%%--Spectrum--and--dynamics--%%%%%%%%%%%%%%%%%%%%%%%%%%%%%%%%%%%%%%%%%%%%%%%%%%%%%%%%%%%%%%%%%%%%%%%%%%%% 

\subsection{Spectrum and dynamics} 

\noindent We recall that, for every $\xi \in \cal{H}$, the one-parameter unitary group ${\mathbb{R}} \ni t \mapsto e^{-itT}$ is so that the curve $e^{-itT}\xi$, in some sense, solves (\ref{SE}). Next we list two quantities usually considered to probe the large time behaviour of the dynamics $e^{-itT}\xi$. 

\begin{enumerate}
\item Let $A$ be a positive operator such that, for each ~$t \in \mathbb{R}$, ~$e^{-itT}{\cal{D}}(A) \subset {\cal{D}}(A)$. For each $\xi \in  {\cal{D}}(A)$, the expectation value of $A$ in the state $\xi$ at time $t$ is defined as

\begin{A}\label{A}
A_\xi^T := \langle e^{-itT}\xi, Ae^{-itT}\xi \rangle.    
\end{A}
\item The time average for $A$ is defined as

\begin{B}\label{B}
\langle A_\xi^T \rangle_t := \frac{1}{t} \int_0^t \langle e^{-isT}\xi, Ae^{-isT}\xi \rangle \, ds.    
\end{B}
\end{enumerate}

In this paper, we are interested in studying the relations between the large time behaviour of the time average of expectation value of $A$ in the state $\xi$ and the spectral properties of the spectral measure $\mu_\xi^T$ of $T$ associated with $\xi$. In this context, firstly we refer to notorious RAGE Theorem, named after Ruelle, Amrein, Georgescu, and Enss \cite{Oliveira}.

\begin{theorem}[RAGE Theorem] Let $A$ be a compact operator on $\cal{H}$. For every $\xi \in \cal{H}$,
\[\displaystyle\lim_{t \to \infty} \langle \vert A_\xi^T \vert \rangle_t = 0,\]
if, and only if, $\mu_\xi^T$ is purely continuous.
\end{theorem}

We note that since any projector onto a finite-dimensional subspace of $\cal{H}$ satisfies the hypotheses of RAGE Theorem, initial states whose spectral measures are purely continuous can be interpreted as those whose trajectories escape, in time average, from every finite-dimensional subspace.

%%%%%%%%%%%%%%%%%%%%%%%%%%%%%%%%%%%%%%%%%%%%%%%%%%%%%%%%%%%%%%%%%%%%%--Fine--scales--of--decay--%%%%%%%%%%%%%%%%%%%%%%%%%%%%%%%%%%%%%%%%%%%%%%%%%%%%%%%%%%%%%%%%%%%%%%%%%%%% 

\subsection{Fine scales of decay} 

\noindent A complete metric space $(X,d)$ of self-adjoint operators, acting in the separable complex Hilbert space $\mathcal{H}$, is said to be regular if convergence with respect to~$d$ implies strong resolvent convergence of operators (see Definition \ref{convergencedef} ahead). One of the results stated in~\cite{Simon}, the so-called Wonderland Theorem, says that, for some regular spaces~$X$, $\{T \in X \mid T$ has purely singular continuous spectrum$\}$ is a dense $G_\delta$ set in $X$. Hence, for these spaces, by RAGE Theorem, for each compact operator $A$ and each $\xi \in \cal{H}$,  $\{T \in X \mid \displaystyle\lim_{t \to \infty} \langle \vert A_\xi^T \vert \rangle_t = 0 \}$ contains a dense $G_\delta$ set in $X$. In this work we present refinements of this result for three different classes of self adjoint operators. 

\begin{enumerate}

\item {\bf General operators.} Let $a>0$ and endow $X_a:= \{ T \mid T$  self-adjoint, $\Vert T\Vert  \leq a \}$ with the metric 
\[d(T,T'):= \sum_{j=0}^\infty \min (2^{-j},\Vert (T-T')e_j\Vert),\]
where $(e_j)$ is an orthonormal basis of $\cal{H}$. Then, $X_a$ is a complete metric space such that metric convergence implies strong resolvent convergence.

\ 

\item {\bf Jacobi matrices.} For every fixed $b > 0$, let the family of Jacobi matrices, $M$, given on ${\ell}^2({{\mathbb{Z}}})$ by the action
\[(Mu)_j := u_{j-1} + u_{j+1} + v_ju_j,\]
where $(v_j)$ is a sequence in ${\ell}^\infty({\mathbb{Z}})$, such that, for each $j \in {\mathbb{Z}}$, $\vert v_j \vert \leq b$. Denote by $X_b$ the set of
these matrices endowed with the topology of pointwise convergence on $(v_j)$. Then, one has that $X_b$ is (by Tychonoff's Theorem) a compact metric space so that metric convergence implies strong resolvent convergence.

\ 

\item  {\bf Schr\"{o}dinger operators.} Fix $C>0$ and consider the family of Schr\"odinger operators, $H_V$,  defined in the Sobolev space ${\mathcal{H}}^2({{\mathbb{R}}})$ by the action 
\[(H_V u)(x) := -\triangle u(x) + V(x)u(x),\]
with  $V \in {\cal{B}}^\infty({{\mathbb{R}}})$ (the space of bounded Borel functions) so that, for each $x \in {\mathbb{R}}$, $ \vert V(x) \vert \leq C$. Denote by $X_C$ the set of these operators endowed with the metric
\[d(H_V,H_U):=\displaystyle\sum_{j= 0}^\infty \min(2^{-j},\Vert V-U\Vert_j),\]
 where $\Vert V-U\Vert_j:= \displaystyle\sup_{x\in B(0,j)}\vert V(x)-U(x)\vert$. Then, one has that $X_C$ is (again by Tychonoff's Theorem) a compact metric space such that convergence metric implies strong resolvent convergence (Definition \ref{convergencedef}). Namely, if $H_{V_k} \rightarrow H_V$ in $X_C$, then, for every $x \in {\mathbb{R}}$, one has that $\displaystyle\lim_{k \to \infty} V_k(x) = V(x)$. Therefore, for each $u \in {\mathrm L}^2({\mathbb{R}})$, by the second resolvent identity and  dominated convergence, 
\begin{eqnarray*}\| (R_i(H_{V_k}) - R_i(H_{V}))u \|_{{\mathrm L}^2({\mathbb{R}})} &=& \| R_i(H_{V_k})(V_k - V)R_i(H_V)u \|_{{\mathrm L}^2({\mathbb{R}})}\\
&\leq& \|(V_k - V)R_i(H_V)u \|_{{\mathrm L}^2({\mathbb{R}})} \longrightarrow 0
\end{eqnarray*}
as $k \rightarrow \infty$.

\end{enumerate}

Next, we introduce a dynamic quantity, the time average of quantum return probability, and its lower and upper decaying exponents, respectively, given by \cite{Barbaroux1} 

\begin{equation}\label{prob1}
\liminf_{t \to \infty} \frac{\log \biggr(\frac{1}{t}\int_0^t |\langle \xi, e^{-isT} \xi \rangle|^2 ds \biggr)}{\log t} = -D_2^+(\mu_\xi^T),  
\end{equation}

\begin{equation}\label{prob2}
\limsup_{t \to \infty} \frac{\log \biggr(\frac{1}{t}\int_0^t |\langle \xi, e^{-isT} \xi \rangle|^2 ds \biggr)}{\log t}  = -D_2^-(\mu_\xi^T),    
\end{equation}
where $D_2^-(\mu_\xi^T)$ and $D_2^+(\mu_\xi^T)$ denote, respectively, the lower and upper correlation dimensions of $\mu_\xi^T$ (see Theorem 2.2 in \cite{Barbaroux1}).

For every $\{e_n\}_{n \in {\mathbb{Z}}}$ orthonormal basis of~$\mathcal{H}$ and every $q>0$ we also introduce the (time-average) $q$-moment of the position operator at time $t>0$, with initial condition~$\xi$, 

\[\langle \langle |X|^q \rangle \rangle_{t,\xi} := \frac{1}{t} \int_0^t \sum_{n \in {\mathbb{Z}}} |n|^q |\langle e^{-isT}\xi ,e_n \rangle|^2~{\mathrm d}s.\]
Let the lower and upper transport exponents, respectively, given by

\begin{equation}\label{moments1}
\beta^-(\xi,q) :=  \liminf_{t \to \infty} \frac{ \log \langle \langle |X|^q \rangle \rangle_{t,\xi} }{ q \log  t},    
\end{equation}

\begin{equation}\label{moments2}
\beta^+(\xi,q) :=  \limsup_{t \to \infty} \frac{ \log \langle \langle |X|^q \rangle \rangle_{t,\xi} }{ q \log  t}.    
\end{equation}

In \cite{Carvalho1,Carvalho2} Carvalho and de Oliveira has been proven refinements of Wonderland Theorem for some classes of discrete self-adjoint operators; in particular, they have shown, for the spaces $X_a$ and $X_b$, that the set of operators whose lower and upper correlation dimensions are simultaneously zero and one (this corresponds to the distinct polynomially decaying rates for the  quantum return probability (\ref{prob1})-(\ref{prob2})), respectively, is generic \cite{Carvalho1} and that is true a dual result  for the transport exponents (which implies in distinct polynomially growth rates for the $q$-moments (\ref{moments1})-(\ref{moments2})) \cite{Carvalho2}. Thus, for $X_a$ and $X_b$ (from the topological viewpoint) the phenomenon of quantum intermittency is exceptional. In this paper we go beyond. Namely, we presented results that show this phenomenon, for $X_a$ and $X_b$, through of a more robust dynamic quantity (\ref{B}) than these discussed in \cite{Carvalho1,Carvalho2}; more specifically, thanks to the RAGE Theorem, we evaluate arbitrary decaying rates of (\ref{B}) (Theorem \ref{maintheorem1}). Moreover, in this work, we also developed an argument involving separability to extend partially such results to the class of (continuous) Schr\"{o}dinger operators $X_C$ (Theorem \ref{maintheorem2}).

We shall prove following 

\begin{theorem}\label{maintheorem1} Let $\alpha : {\mathbb{R}} \longrightarrow {\mathbb{R}}$ such that 
\[\limsup_{t \to \infty} \alpha(t) = \infty.\]
For every compact operator $A$ and every $\xi \in {\cal{H}}$ non null,
\begin{eqnarray*}
\{ T \in X_\Gamma \mid T ~has~purely~singular~continuous~spectrum,\\
\limsup_{t \to \infty} \alpha(t) \langle \vert A_\xi^T \vert \rangle_t = \infty~and~\liminf_{t \to \infty} t \langle \vert A_\xi^T \vert \rangle_t = 0 \}
\end{eqnarray*}
is a dense $G_\delta$ set in $X_\Gamma$ for $\Gamma \in \{a,b\}$.
\end{theorem}

\begin{remark}\label{remarkmaintheorem1}
\end{remark}

\begin{enumerate}

\item Theorem~\ref{maintheorem1} one says that, for $X_a$ and $X_b$, typically Baire's sense, expectation values of compact operators, in time average, have distinct decaying rates; in particular, in this case, typically, the rates with whom the trajectories escape, in time average, from every finite-dimensional subspace, depend on subsequences of time going to infinite. We emphasize that $\alpha$, in statement of Theorem~\ref{maintheorem1},  can be chosen arbitrarily and that is well known that, about rather general conditions on $A$, $T$ and $\xi$, for every $\epsilon>0$,    
\[\liminf_{t \to \infty} t^{1+\epsilon} \langle \vert A_\xi^T \vert \rangle_t = \infty.\]    

\item The proof from that 
\[\{ T \in X_\Gamma \mid \liminf_{t \to \infty} t \langle \vert A_\xi^T \vert \rangle_t = 0 \}\]
is a dense $G_\delta$ set in $X_\Gamma$ is a direct application of (\ref{prob1}) and (\ref{prob2}) (Theorem 2.2 in \cite{Barbaroux1}), Lemma 3.2 and Theorem 3.2 (Theorem \ref{theoLast}) in \cite{Last} and Theorems 1.2 and 1.4 in \cite{Carvalho1}. For the convenience of the reader, we presented in details a simple proof of this fact here.

\item The proof from that 
\[\{ T \in X_\Gamma \mid \limsup_{t \to \infty} \alpha(t) \langle \vert A_\xi^T \vert \rangle_t = \infty \}\]
is a dense $G_\delta$ set in $X_\Gamma$ is more delicate since that involves an arbitrary growth of $\alpha$. To prove such result we use the RAGE Theorem combined with the fact that, for some $\gamma(\Gamma)>0$, 
\[D_\gamma := \{T \in X_\Gamma \mid \sigma(T) =[-\gamma,\gamma]  {\rm ~is~pure~point}\}\]
is a dense set in $X_\Gamma$. For $X_\Gamma = X_a$, is a direct consequence from the Weyl-von Neumann Theorem \cite{von Neumann,Weyl} that $D_a$ is a dense set in $X_a$ (see proof of Theorem in 3.2 \cite{Simon} for details), whereas for $X_\Gamma = X_b$ is known that this can be obtained by using Anderson's localization. Namely, for every fixed $b > 0$, consider $\Omega = [-b,b]^{\mathbb{Z}}$ be endowed with the product topology and with the respective Borel $\sigma$-algebra. Assume that $(\omega_j)_{j \in \mathbb{Z}} = \omega \in \Omega$ is a set of independent, identically distributed real-valued random variables with a common probability measure $\rho$ not concentrated on a single point and so that $\int \vert \omega_j \vert^\theta {\mathrm d}\rho(w_j) < \infty$ for some $\theta>0$. Denote by $\nu:=\rho^{\mathbb{Z}}$ the probability measure on $\Omega$. The Anderson model is a random Hamiltonian on $\ell^2(\mathbb{Z})$, defined for each $\omega \in \Omega$ by
\[(h_\omega u)_j :=  u_{j-1} + u_{j+1} + \omega_j u_j.\]
It turns out that \cite{Aizenman,DamanikDDP,StolzIntr}  
\[\sigma(h_\omega) = [-2,2] + \text{\rm supp}(\rho),\]
and $\nu$-a.s. $\omega$, $h_\omega$ has pure point spectrum~\cite{carmona,von} (see also Theorem 4.5 in \cite{DamanikDDP}). Thus, if $\mu$ denotes the product of infinite copies of the normalized Lebesgue measure on $[-b,b]$, that is, $(2b)^{-1}\ell$, then 
\[D_{b+2} =\{M \in X_b \mid \sigma(M) = [-b-2,b+2], ~\sigma(M) \text{ is pure point}\}\]
is so that $\mu (X_b \backslash D_{b+2}) =0$.

\item We note that for $X_C$ the ideas presented above do not apply. Namely, this comes from the fact that this is a space of unbounded operators and, to the best of our knowledge, still no has been detailed in the literature arguments that show that 
\[\{H \in X_C \mid \sigma(H)  {\rm ~is~pure~point}\}\]
 is a dense set in $X_C$. In this context, in this work, we use the separability of $X_C$ to prove the following result. 
\end{enumerate}

\begin{theorem}\label{maintheorem2} Take $\alpha$ be as before. Then, for every compact operator $A$, there exists a dense $G_\delta$ set $G_\alpha(A)$ in ${\mathrm L}^2({\mathbb{R}})$ such that, for every $\xi \in G_\alpha(A)$, 
\begin{eqnarray*}
\{ H \in X_C \mid H ~has~purely~singular~continuous~spectrum~on~(0,\infty)\\
and~\displaystyle\limsup_{t \to \infty} \alpha(t) \langle \vert A_\xi^{H} \vert \rangle_t = \infty\}
\end{eqnarray*}
is a dense $G_\delta$ set in $X_C$.
\end{theorem}

The paper is organized as follows. In Section \ref{secGJ}, we present the proof of Theorem \ref{maintheorem1}. In section \ref{secSO}, we prove Theorem \ref{maintheorem2}.

Some words about notation: $\mathcal{H}$ denotes a separable complex Hilbert space. If~$A$ is a linear operator in~$\mathcal{H}$, we denote its domain by ${\mathcal{D}}(A)$, its spectrum by $\sigma(A)$, its point spectrum by $\sigma_p(A)$. If $\varrho(A)$ denotes its resolvent set, then the resolvent operator of $A$ at $\lambda \in \varrho(A)$ is denoted by $R_\lambda(A)$. For every set $\Omega \subset {\mathbb{R}}$, $\chi_\Omega$ denotes the characteristic function of the set $\Omega$. Finally, for every $x \in {\mathbb{R}}$ and $r>0$, $B(x,r)$ denotes the open interval $(x-r,x+r)$. 

%%%%%%%%%%%%%%%%%%%%%%%%%%%%%%%%%%%%%%%%%%%%%%%%%%%%%%%%%%%%--General--operators--and--Jacobi--Matrices--%%%%%%%%%%%%%%%%%%%%%%%%%%%%%%%%%%%%%%%%%%%%%%%%%%%%%%%%%%%%%%%%%%%%%%%%%%%%%%%%%%%%%%%%%%%%%%%%%%%%%%%%%%%%%%%%%%%%%%%%%%%%%%%%%%%%%%%%%%%%%%%%%%%%%%%%%%%%%%%%%%%%%%%%%%%%%%%%%%%%%%%%%%%%%%%%%%%%%%%%%%%%%%%%%%%%%%%%%%%%%%%%%%%%%%%%%%%%%%%

\section{General operators and Jacobi Matrices}\label{secGJ}

\noindent We need some preliminaries to present the proof of Theorem \ref{maintheorem1}. 

\begin{definition}{\rm A sequence of bounded linear operators $(T_n)$ strongly converges to $T$ in~$\cal{H}$ if, for every $\xi \in \cal{H}$, $T_n \xi \longrightarrow T \xi$ in $\cal{H}$.}
\end{definition}

Next, we present definitions of a sequence of (unbounded) self-adjoint operators~$(T_n)$ approaching another one~$T$.

\begin{definition}\label{convergencedef}{\rm Let $(T_n)$ be a sequence of self-adjoint operators and let $T$ be another self-adjoint operator. One says that: 
\begin{enumerate}
\item $T_n$~ converges to ~$T$ in the strong resolvent sense (SR) if $R_i(T_n)$ strongly converges to $R_i(T)$.
\item $T_n$~ converges to ~$T$ in the strong dynamical sense (SD) if, for each $t \in \mathbb{R}$, $e^{itT_n}$ strongly converges to $e^{itT}$.
\end{enumerate}}
\end{definition}

The next result shows that both notions of convergence are equivalent.

\begin{proposition}[Proposition 10.1.9 in \cite{Oliveira}]\label{SR-SDproposition} The SR and SD convergences of self-adjoint operators are equivalent.
\end{proposition}

\begin{definition}{\rm Let $\mu$ be a $\sigma$-finite positive Borel measure on $\mathbb{R}$. One says that $\mu$ is Lipschitz continuous if there exists a constant $C>0$ such that, for each interval $I$ with $\ell(l) < 1$, $\mu(I) < C\, \ell(I)$, where $\ell(\cdot)$ denotes the Lebesgue measure on~$\mathbb{R}$.}
\end{definition}

We need the following 

\begin{theorem}[Theorem 3.2 in \cite{Last}]\label{theoLast} If $\mu_\xi^T$ is Lipschitz continuous, then there exists a constant $C_\xi$ such that for any compact operator $A$ and $t>0$
\[\langle \vert A_\xi^T \vert \rangle_t \, < \, C_\xi \Vert A \Vert_1 t^{-1}, \]
where $\Vert A \Vert_1$ denotes the trace norm of $A$.
\end{theorem}

Now we confirm the phenomenon of quantum intermittency for regular spaces of self-adjoint operators about mild hypotheses.

\begin{proposition}\label{GMprop}  Let $\alpha$ as in the statement of Theorem \ref{maintheorem1} and let $\xi \in \cal{H}$ non null. Let $X$ be a regular space of self-adjoint operators. Suppose that: 

\begin{enumerate}

\item For some $\gamma>0$,  $D_\gamma := \{T \in X \mid \sigma(T) =  [-\gamma,\gamma]  ~is  ~pure  ~point\}$ is a dense set in $X$.

\item $L := \{T \in X \mid \mu_\xi^T ~is ~Lipschitz  ~continuous\}$ is a dense set in $X$. 

\end{enumerate}
Then, for any compact operator $A$,
\[\{ T \in X \mid  \displaystyle\limsup_{t \to \infty} \alpha(t) \langle \vert  A_\xi^T \vert \rangle_t = \infty \, and \, \displaystyle\liminf_{t \to \infty} t \langle \vert A_\xi^T \vert \rangle_t = 0 \}\] 
is a dense $G_\delta$ set in $X$.
\end{proposition}

\begin{proof} Since, for each $t \in {\mathbb{R}}$, by Proposition \ref{SR-SDproposition} and dominated convergence, the mapping 
\[X \ni T \mapsto \alpha(t) \langle \vert A_\xi^T \vert \rangle_t\]
is continuous, it follows that, for each $k \geq 1$ and each $n \geq 1$, the set
\[\bigcup_{t \geq k} \{ T \in X \mid \alpha(t) \langle \vert A_\xi^T \vert \rangle_t > n \} \]
is open, from which follows that 
\[\{ T \in X \mid  \displaystyle\limsup_{t \to \infty} \alpha(t) \langle \vert  A_\xi^T \vert  \rangle_t = \infty\} = \bigcap_{n \geq 1} \bigcap_{k \geq 1} \bigcup_{t \geq k} \{ T \in X \mid \alpha(t) \langle \vert A_\xi^T \vert  \rangle_t > n \} \]
is a $G_\delta$ set in $X$. By RAGE Theorem, $D_\gamma \subset \{ T \in X \mid  \displaystyle\limsup_{t \to \infty} \alpha(t) \langle \vert A_\xi^T \vert \rangle_t = \infty\}$. Hence, $\{ T \in X \mid  \displaystyle\limsup_{t \to \infty} \alpha(t) \langle \vert  A_\xi^T \vert \rangle_t = \infty\}$ is a dense $G_\delta$ set in $X$.

We note that, for each $j \geq 1$, 
\[L_j := \{ T \in X \mid \displaystyle\liminf_{t \to \infty} t^{1-\frac{1}{j}} \langle \vert A_\xi^T \vert \rangle_t = 0 \}\] 
is also a $G_\delta$ set in $X$. Since, by Theorem \ref{theoLast}, for each $j \geq 1$, $L \subset L_j$, it follows that, by Baire's Theorem, 
\[\{ T \in X \mid \displaystyle\liminf_{t \to \infty} t \langle \vert  A_\xi^T \vert  \rangle_t = 0 \} =  \bigcap_{j \geq 1} L_j\]
is a  dense $G_\delta$ set in $X$, concluding the proof of proposition.
\end{proof}

\begin{proof} [{Proof} {\rm (Theorem~\ref{maintheorem1})}] As each $T \in X_\Gamma$ can be approximated by an operator whose spectral measures are Lipschitz continuous (see proof of Theorems 1.2 and 1.4 in \cite{Carvalho1} for details). The theorem is now a direct consequence of Remark \ref{remarkmaintheorem1}, Proposition \ref{GMprop}, Theorems 3.1 and 4.1 in \cite{Simon} and Baire's Theorem.
\end{proof}

%%%%%%%%%%%%%%%%%%%%%%%%%%%%%%%%%%%%%%%%%%%%%%%%%%%%%%%%%%%%%%%%%%%%%--Schr\"{o}dinger--operators--%%%%%%%%%%%%%%%%%%%%%%%%%%%%%%%%%%%%%%%%%%%%%%%%%%%%%%%%%%%%%%%%%%%%%%%%%%%%%%%%%%%%%%%%%%%%%%%%%%%%%%%%%%%%%%%%%%%%%%%%%%%%%%%%%%%%%%%%%%%%%%%%%%%%%%%%%%%%%%%%%%%%%%%%%%%%%%%%%%%%%%%%%%%%%%%%%%%%%%%%%%%%%%%%%%%%%%%%%%%%%%%%%%%%%%%%%%%%%%%%%%%%%%%%%

\section{Schr\"{o}dinger operators}\label{secSO}

\noindent In order to prove Theorem~\ref{maintheorem2} we need of result to follow and a fine separability argument (check below).

\begin{lemma}\label{lemma}Let $T$ be a self-adjoint operator so that $\sigma_p(T) \not = \emptyset$ and let $\alpha$ as in the statement of Theorem \ref{maintheorem1} . Then, for any compact operator $A$
\[G_\alpha(A,T) := \{ \xi \in {\cal{H}} \mid \limsup_{t \to \infty} \alpha(t) \langle \vert A_\xi^T \vert \rangle_t = \infty\}\]
is a  dense $G_\delta$ set in $\cal{H}$.
\end{lemma}

\begin{proof} Since, for each $t \in {\mathbb{R}}$, by dominated convergence, the mapping 
\[{\cal{H}} \ni \xi \mapsto \alpha(t) \langle \vert A_\xi^T \vert  \rangle_t\]
is continuous, it follows that
\[G_\alpha(T,A) = \bigcap_{n \geq 1} \bigcap_{k \geq 1} \bigcup_{t \geq k} \{ \xi \in {\cal{H}} \mid \alpha(t) \langle \vert  A_\xi^T \vert \rangle_t  > n \} \]
is a $G_\delta$ set in $\cal{H}$.

Given $\xi  \in \cal{H}$, write $\xi = \xi_1 + \xi_2$, with $\xi_1 \in Span\{\xi_0\}^\perp$ and $\xi_2 \in Span\{\xi_0\}$, where $\xi_0$, with $\Vert \xi_0 \Vert_{\cal{H}} =1$, is an eigenvector of $T$ associated with an eigenvalue $\lambda$. If $\xi_2 \not =  0$, then 
\begin{eqnarray*}
\nonumber \mu_\xi^T (\{\lambda\}) &=&  \Vert P^T(\{\lambda\})\xi \Vert_{\cal{H}}^2\\  \nonumber &\geq& 2Re\langle P^T(\{\lambda\})\xi_1,P^T(\{\lambda\})\xi_2 \rangle + \Vert P^T(\{\lambda\})\xi_2 \Vert_{\cal{H}}^2\\ &=&  \Vert \xi_2 \Vert_{\cal{H}}^2 > 0, 
\end{eqnarray*}
where $P^T(\{\lambda\})$ represents the spectral resolution of $T$ over the set $\{\lambda\}$. Now, if $\xi_2 =  0$, define, for each $k\geq1$,
\[\xi_k  := \xi + \frac{\xi_0}{k}.\]
It is clear that $\xi_k \rightarrow \xi$. Moreover, by the previous arguments, for each $k\geq1$, one has
\[ \mu_{\xi_k}^T (\{\lambda\}) > 0.\]
Thus, $G:=\{\xi \in {\cal{H}} \mid \mu_{\xi}^T$ has point component$\}$ is a dense set in $\cal{H}$. Nevertheless, by RAGE Theorem, $G \subset G_\alpha(T,A)$. Hence, $G_\alpha(T,A)$ is a dense $G_\delta$ set in $\cal{H}$.
\end{proof}

\begin{proof}[{Proof} {\rm (Theorem~\ref{maintheorem2})}] By the arguments presented in the proof of Proposition~\ref{GMprop}, for every $\xi \in {\mathrm L}^2({\mathbb{R}})$,
\[\{ H_V \in X_C \mid  \displaystyle\limsup_{t \to \infty} \alpha(t) \langle \vert A_\xi^{H_V} \vert \rangle_t = \infty\}\] 
is a $G_\delta$ set in $X_C$.

Now given $H_V \in X_C$, we define, for every $k \geq 1$,
\[V_k(x)  := \frac{k}{k+1}\chi_{B(0,k)} V(x) - \frac{C}{(k+1)(\vert x \vert+1)}.\]
We note that, for each $k \geq 1$ and each $x \geq k$, $V_k(x) \leq 0$. Moreover, for each $k \geq 1$,
\[\int_k^\infty V_k(x) = -\infty.\]
Therefore, for every $k \geq 1$, $H_{V_k}$ has at least a negative eigenvalue \cite{Schechter}; in particular, $\sigma_p(H_{V_k}) \not = \emptyset$. Since $H_{V_k} \rightarrow H_V$ in $X_C$, it follows that 
\[Y := \{H_V \in X_C \mid \sigma_p(H_V) \not = \emptyset  \}\] 
is a dense set in $X_C$.

Now, let $(H_{V_k})$ be an enumerable  dense subset  in $Y$ (which is separable, since $X_C$ is separable); then,  by Lemma \ref{lemma} and Baire's Theorem, $\bigcap_{k\geq 1} G_\alpha(H_{V_k},A)$  is a dense $G_\delta$ set in ${\mathrm L}^2({\mathbb{R}})$. Moreover, for every $\xi \in \bigcap_{k\geq 1} G_\alpha(H_{V_k},A)$, 
\[\{ H_V \in X_C \mid  \displaystyle\limsup_{t \to \infty} \alpha(t) \langle \vert A_\xi^{H_V} \vert \rangle_t = \infty\} \supset \bigcup_{k\geq 1} \{ H_{V_k} \}\] 
is a dense $G_\delta$ set in $X_C$. The theorem is now a consequence of Theorem 4.5 in \cite{Simon} and Baire's Theorem.
\end{proof}

\begin{remark}{\rm Note as this separability argument used in the proof of Theorem \ref{maintheorem1}, in a sense, has allowed to use a typical behaviour in ${\mathrm L}^2({\mathbb{R}})$ (Lemma \ref{lemma}) to find a typical behaviour in $X_C$.}
\end{remark}

%%%%%%%%%%%%%%%%%%%%%%%%%%%%%%%%%%%%%%%%%%%%%%%%%%%%%%%%%%%%%%%%%%%%%%%%%%--Acknowledgments--%%%%%%%%%%%%%%%%%%%%%%%%%%%%%%%%%%%%%%%%%%%%%%%%%%%%%%%%%%%%%%%%%%%%%%%%%%%%%%%%%%%%%%%%%%%%%%%%%%%%%%%%%%%%%%%%%%%%%%%%%%%%%%%%%%%%%%%%%%%%%%%%%%%%%%%%%%%%%%%%%%%%%%%%%%%%%%%%%%%%%%%%%%%%%%%%%%%%%%%%%%%%%%%%%%%%%%%%%%%%%%%%%%%%%%%%%%%%%%%%%%%%%%%%%%%%%%%%

\begin{center} \Large{Acknowledgments} 
\end{center}

M.A. was supported by CAPES (a Brazilian government agency; Contract 88882.184169/2018-01). The author is grateful to César R. de Oliveira and Silas L. Carvalho for fruitful discussions and helpful remarks. 

%%%%%%%%%%%%%%%%%%%%%%%%%%%%%%%%%%%%%%%%%%%%%%%%%%%%%%%%%%%%%%%%%%%%%%%%%%%%%%%%%%%%%%%%%%%%%%%%%%%%%%%%%%%%%%%%%%%%%%%%%%%%%%%%%%%%%%%%%%%%%%%%%%%%%%%%%%%%%%%%%%%%%%%%%%%%%%%%%%%%%%%%%%%%%%%%%%%%%%%%%%%%%%%%%%%%%%%%%%%%%%%%%%%%%%%%%%%%%%%%%%%%%%%%%%--thebibliography--%%%%%%%%%%%%%%%%%%%%%%%%%%%%%%%%%%%%%%%%%%%%%%%%%%%%%%%%%%%%%%%%%%%%%%%%%%%%%%%%

\noindent  Email: ec.moacir@gmail.com, Departamento de Matem\'atica, ~UFMG, Belo Horizonte, MG, 30161-970 Brazil

\end{document}